\documentclass[reqno,11pt]{amsart}
\usepackage{amssymb,amsmath,amsfonts}
\usepackage[bookmarks=true, colorlinks=true, linkcolor=red, citecolor = red, menucolor = red, urlcolor=blue]{hyperref}
\usepackage{fancyhdr}
\usepackage{amssymb}
\usepackage{amsmath,amssymb}
\usepackage[latin1]{inputenc}
\newtheorem{thm}{Theorem}[section]
\newtheorem{defi}[thm]{Definition}
\newtheorem{lem}[thm]{Lemma}
\newtheorem{cor}[thm]{Corollary}
\newtheorem{prop}[thm]{Proposition}
\newtheorem{rmk}[thm]{Remark}

\usepackage{enumerate}
\usepackage{indentfirst}
\usepackage{esint}
\everymath{\displaystyle}

\begin{document}

\title[Heat equation with degenerate coefficients]{Existence and non-existence of global solutions for a heat equation with degenerate coefficients}

\author{ Ricardo Castillo}
\address{Departamente de Matem\'atica, Facultad de Ciencias, Universidad del B\'io-B\'io - Chile}
\email{rcastillo@ubiobio.com}

\author{Omar Guzm\'an-Rea}
\address{Departamente de Matem\'atica, Universidade de Bras\'ilia - Brazil}
\email{omar.grea@gmail.com}

\author{Mar\'ia Zegarra}
\address{Departamente de Matem\'atica, Universidad Nacional Mayor de San Marcos - Per\'u}
\email{mzegarrag@unmsms.edu.pe}


\begin{abstract}
In this paper, we will study the following parabolic problem $u_t - div(\omega(x) \nabla u)= h(t) f(u) + l(t) g(u)$ with non-negative initial conditions pertaining to $C_b(\mathbb{R}^N)$, where the weight $\omega$ is an appropriate function that belongs to the Munckenhoupt class $A_{1 + \frac{2}{N}}$ and the functions $f$, $g$, $h$ and $l$ are non-negative and continuous. The main goal is to establish of global and non-global existence of non-negative solutions. In addition, to present the particular case when $h(t) \sim t^r ~~ (r>-1)$, $l(t) \sim t^s ~~ (s>-1)$, $f(u) = u^p$ and $g(u)= (1+u)[\ln(1+u)]^p,$ we obtain both the so-called Fujita's exponent and the second critical exponent in the sense of Lee and Ni \cite{Lee-Ni}. Our results extend those obtained by Fujishima et al. \cite{Fujish} who worked when $h(t)=1$, $l(t)=0$ and $f(u)=u^p $.



\end{abstract}
\maketitle

\section{Introduction}
Let $T>0$ and $N \geq 1$. We consider the following heat equation
\begin{equation}\label{eq2.1}
 \left\{\begin{array}{rll}
 u_{t} - \mbox{div}  (\omega(x) \nabla u )  = & h(t) \, f(u) + l(t) \, g(u)& (x,t) \in  \mathbb{R}^N \times (0,T),\\
 u(0)= & u_{0}\geq 0 & x \in  \mathbb{R}^N ,
      \end{array}\right.
\end{equation}
where $u_{0}\in C_b(\mathbb{R}^N)$, $h,l: [0,\infty) \rightarrow [0,\infty)$ are continuous functions, the functions $f,g: [0, \infty) \rightarrow [0,\infty)$ are locally Lipschitz continuous and the spatial function $\omega : \mathbb{R}^N \rightarrow [0,\infty)$ is a weight, which can be
\begin{itemize}
 \item[(A)] $\omega(x) = \ \mid x_1 \mid^{a}$ with $a \in [0, 1)$ if $N = 1,2$; and $a \in [0, 2/N)$ if $N \geq 3,$ \\
or
  \item[(B)] $\omega(x) = \ \mid x \mid^{b}$ with $b \in [0, 1)$.
\end{itemize}
Here $x_1$ is the first coordinate of $x = (x_1,...,x_N) \in \mathbb{R}^N$ and $ \mid . \mid$ is the Euclidean norm of $\mathbb{R}^N$. Note that the operator $ \mbox{div}(\omega(x) \nabla(u)) $ under the conditions $(A)$ or $(B) $ on $ \omega(x)$ is not self-adjoint. Fujishima et al. very well commented on this aspect (see \cite[p.6]{Fujish}). Something important to note is that the degenerate operator $\mbox{div}  (\omega(x) \nabla u )$ when $\omega(x) = \mid x_1 \mid^{a},$ is related to the fractional Laplacian, through the Caffarelli-Silvestre extension (see \cite{Caffarelli}, \cite{Sande} and \cite[p.3]{Fujish}). The fractional Laplacian has found extensive applications in several physical and biological models with anomalous diffusion; for example, see \cite{Klafter} and the references therein. 


Equation \eqref{eq2.1} appears in models that describe the processes of heat propagation in inhomogeneous media, see e.g. \cite{Kamin3}, \cite{Kamin4}, \cite{Kamin5}, \cite{Zeldovich} where the authors studied thermal phenomena related to the following equation
\begin{equation}\label{A}
v_t - \mid x \mid^{\beta_1} \mbox{div}( \mid x\mid^{\beta_2} \nabla v^m )=0.
\end{equation}
Recently in \cite{Bonforte}, M. Bonforte and N. Simonov obtained a priori Harnack inequalities and H\"{o}lder continuity for a class of nonnegative local weak solution of (\ref{A}), where the weights $\mid x \mid^{\beta_1} $ and $\mid x \mid^{\beta_2}$ does not necessarily belong to the Muckenhoupt class $A_{2}$ and $m \geq 1$.

Several authors extensively studied the problem \eqref{eq2.1} when $\omega(x) = 1 $. Let us mention, for instance, the seminal work of Fujita \cite{FU}, who studied problem \eqref{eq2.1}, and showed the existence of the following value $p^\star(0):=1 + 2/N$ known as critical Fujita exponent, which determines both the global and the non-global existence of nonnegative solutions. Specifically, He proved the following result: Assume that $\omega(x)=1 $, $l(t) =0$, $h(t)=1 $, and $f(u)=u^p $ with $p>1. $
\begin{itemize}
  \item If $1 < p < p^\star(0)$, then problem (\ref{eq2.1}) does not admit any nontrivial nonnegative global solution.
  \item If $p > p^\star(0)$, then problem (\ref{eq2.1}) has global solutions, depending on the size of the initial condition.
\end{itemize}

See also \cite{Hayakawa}, \cite{AW}, \cite{QS} for more details in these settings. Another significant result in this direction is that obtained by Lee and Ni in \cite{Lee-Ni}, who determined the so-called second critical exponent $\rho^\star(0)= \frac{2}{p-1} $ when $p^\star(0) < p$ under suitable decay conditions on the initial data. More precisely: Suppose the same previous conditions on (\ref{eq2.1}) with $  p^\star(0)< p $.
\begin{itemize}
  \item If $0 < \rho \leq \rho^\star(0)$, then problem (\ref{eq2.1}) has no global solutions for any $\psi \in I_\rho $.
  \item If $\rho^\star(0) < \rho <N$, then for any $u_0 \in I^{\rho} $ there exists $\Lambda(0) >0 $ such that the problem (\ref{eq2.1}) with initial data $\lambda u_0 $ has a global solution for all $0 < \lambda < \Lambda(0) $.
\end{itemize}
Here
\begin{equation*}
\begin{array}{ll}
I^{\rho} &=   \{ \psi \in C_b (\mathbb{R}^{N}), \psi \geq 0  \mbox{ and }  \limsup_{\mid x \mid \rightarrow \infty}  \mid x \mid^{\rho} \psi(x) < \infty \}, \mbox{ and }\\ \noalign{\medskip}
I_{\rho} &=   \{ \psi \in C_b (\mathbb{R}^{N}), \psi \geq 0  \mbox{ and }  \liminf_{\mid x \mid \rightarrow \infty}  \mid x \mid^{\rho} \psi(x) > 0  \}.
\end{array}
\end{equation*}
for $ \rho >0.$ For other works in this line of research, the readers can refer to \cite{Yang1}, \cite{JYang1}, \cite{JYang2}.


Fujishima et al. \cite{Fujish}, established the first results with optimal conditions for the global existence for a problem related to (\ref{eq2.1}), where $\omega$ either is $(A)$ or is $(B)$, and showed the following Fujita-type result.
\begin{thm}[\cite{Fujish}] \label{THMFU} Assume either (A) or (B). Let $\alpha = a$ for the case (A) or $\alpha = b$ for the case (B). Assume  $h(t) = 1 $, $l(t)=0$, $f(t) =t^p $, and $p^\star(\alpha) = 1 + \frac{2 - \alpha}{N} $.
\begin{itemize}
  \item[(i)] If $1 < p \leq p^\star(\alpha) $, then problem $\eqref{eq2.1} $ has no nontrivial global-in-time solutions.
  \item[(ii)] If $p^\star(\alpha) < p $, then problem $\eqref{eq2.1} $ has nontrivial global-in-time solutions.
\end{itemize}
\end{thm}
For the non-global existence results in Theorem \ref{THMFU}, the authors powerfully used Lemma $3.1$ in \cite{Fujish}, which was obtained through an iterative method, that has been widely used in several other works, e.g. see \cite{Bai}, \cite{Escobedo}, \cite{Miguel}, \cite{RMC}, \cite{RM}. Unfortunately, this same approach cannot be used for the case of problem \eqref{APPL}, due to logarithmic nonlinearity $(u+1) (\ln(u+1))^q $.

In the current work, we are interested in establishing conditions for the global and nonglobal existence of solutions of \eqref{eq2.1}, adapting the ideas found in \cite{Meier1} and \cite{Nosso}. Meier, in \cite{Meier1}, studied the phenomenon of blow-up for the problem \eqref{eq2.1}, when $\omega(x) = 1$, $l(t)=0$ and $f$ is continuously differentiable with $f' \geq 0$ and satisfies the integral condition in $(\Phi_1)$ (this condition is given later). More specifically, he showed that if $f$ is convex with $f(0)=0$, then there is blow-up in finite time for the solution of the problem \eqref{eq2.1}, provided
$\displaystyle\int_{\| T(\tau) u_0 \|_\infty}^\infty \frac{d \sigma}{f(\sigma)} < \int_0^\tau h(\sigma) d\sigma \, \mbox{ for some } \, \tau > 0, $
where $\{ T(t) \}_{t \geq 0}$ is the Dirichlet heat semigroup in a smooth domain $\Omega$ (see \cite[p.440]{QS}). Later, Loayza and C. Da Paix\~{a}o \cite{Nosso} improved those results and established conditions for the global and non-global existence of nonnegative solutions of problem \eqref{eq2.1}, when $l=0$, and $f$ is locally Lipschitz continuous. The authors also applied their results to obtain Fujita-type results (see \cite[Theorem 1.7]{Nosso}); notice that they do not address the critical case $q = 1 + \frac{2(r+1)}{N}.$

We applied our results (see Theorem \ref{thm1}) to obtain both Fujita-type results and the second critical exponent, in the sense of \cite{Lee-Ni}, for a problem related to the following equation with logarithmic nonlinearity
\begin{equation}\label{APPL}
 \left\{\begin{array}{rll}
 u_{t} - \mbox{div}  (\omega(x) \nabla u )  = &  h(t) \, u^p \, + \, l(t) \, (u+1) (\ln(u+1))^q & (x,t) \in  \mathbb{R}^N \times (0,T),\\
 u(0)= & u_{0}\geq 0 & x \in  \mathbb{R}^N,
      \end{array}\right.
\end{equation}
with suitable $h(t) $ and $ l(t)$ (see Corollary \ref{cor1} and Remark \ref{RMK1}). As far as we are aware, results concerned to second critical exponent for problems related to \eqref{APPL} with $h(t)=0$ has not been addressed until now, even in the case when $\omega(x)=1.$ Since the first work done in 1979 by Galaktionov, Kurdyumov, Mikhailnov, and  Samarskii \cite{Galaktionov}, problems with logarithmic nonlinearities have received considerable attention in various research papers due to their multiple applications in physics and applied sciences. Readers can see the following works and the references \cite{Galaktionov1}, \cite{Galaktionov2}, \cite{Galaktionov3}, \cite{Galaktionov4}, \cite{Cheng}, \cite{Deng}, \cite{Ding}, \cite{RFerreira}, \cite{Lian}, \cite{Li}.

We would like to mention that the techniques used in this paper can be applied to treat the following more general problem
$$u_t - div (\omega(x) \nabla u) = \sum_{i=1}^k h_i (t) f_i(u)\,\,\,(k \in \mathbb{N}),$$
where $h_i \in C[0,\infty) $ and $f_i \in C[0,\infty)$ for $i=1,...,k.$

To prepare an accurate statement of the non-global existence result, we assume that $f$ or $g$ lies in
$$\Phi = \left\{\gamma:[0, \infty) \rightarrow [0, \infty) : \gamma \mbox{ satisfies } (\Phi_1) \mbox{ and } (\Phi_2) \right\} $$
where
\begin{itemize}
\item[($\Phi_1 $)] $ \displaystyle \int_{z}^{\infty} \frac{d \sigma}{\gamma(\sigma)} < \infty $ for all $z > 0$ and $ \gamma(S(t)v_0) \leq S(t)\gamma(v_0)$ for all $0 \leq v_0 \in C_b(\mathbb{R}^N) $ and $t>0$.
\item[($\Phi_2 $)] The function $\gamma $ is nondecreasing such that $\gamma(s) > 0$ for all $ s>0$.
\end{itemize}
Here $S(t)\phi(x):= \displaystyle  \int_{\mathbb{R}^N} \, \Gamma(x,y,t) \, \phi(y) \, dy \, $, where $\Gamma(x,y,t) $ is the fundamental solution of
\begin{equation} \label{FUND}
 v_{t}- \mbox{div}(\omega(x) \nabla v )  = 0,~~~~ x \in \mathbb{R}^N, ~~~~ t >0
\end{equation}
with a pole in $(y,0).$ Also, $\Gamma(x,y,t)$ is nonnegative, and continuous for all $(x, t) \neq (y,0),$ see for more details \cite{Gutierrez},  \cite{GutierrezArg}.

The integral condition in $(\Phi_1)$ is related to the Osgood-type condition (see, e.g. \cite{Laister1}), while the second condition in $(\Phi_1)$ is satisfied if, for example, $\gamma $ is convex (see Lemma \ref{13.1}).

Our main result is the following
\begin{thm}\label{thm1} Assume either (A) or (B), and supposse that $f,g : [0,\infty) \rightarrow [0,\infty)$ are locally lipchitz continuous functions such that $f(0)=g(0)=0. $
\begin{itemize}
 \item[(i)] If $f $, $ g$, $f(s)/s$ and $g(s)/s$ are nondecreasing in some interval $(0,m] $, and there exist nontrivial $ 0 \leq v_0 \in C_b(\mathbb{R}^N) $, $ \|v_0\|_\infty \leq m $ satisfies
 \begin{equation}
 \int_{0}^{\infty} h(\sigma) \, \frac{f \left(\|S(\sigma)v_0\|_{\infty} \right)}{\|S(\sigma)v_0\|_{\infty}} \, d\sigma \, + \, \int_{0}^{\infty} l(\sigma) \, \frac{g \left(\|S(\sigma)v_0\|_{\infty} \right)}{\|S(\sigma)v_0\|_{\infty}} \, d\sigma  < 1,
 \end{equation}
 then there exists $ \delta > 0 $, such that for $ \delta v_0 = u_0,$ the solution of \eqref{eq2.1} is a global solution.
 \item[(ii)] Let $0 \leq u_0 \in C_b(\mathbb{R}^N)$, $u_0 \neq 0 $ and suppose that one of the following conditions hold:
\begin{itemize}
  \item[(a)]  $f \in \Phi $ and there exist $ \tau > 0$ such that
  \begin{equation} \label{TTa}
  \int_{ \| S(\tau)u_0 \|_\infty }^{\infty} \frac{d \sigma}{f(\sigma)} \leq \int_{0}^{\tau} h(\sigma) d \sigma
  \end{equation}
  \item[(b)] $g \in \Phi $ and there exist $ \tau > 0$ such that
  \begin{equation} \label{TTb}
  \int_{ \| S(\tau)u_0 \|_\infty }^{\infty} \frac{d \sigma}{g(\sigma)} \leq \int_{0}^{\tau} l(\sigma) d \sigma
  \end{equation}
\end{itemize}
Then the solution of problem \eqref{eq2.1} with initial condition $u_0$ is non-global.
\end{itemize}
\end{thm}

As one application of our Theorem, we yield the following Corollary.
\begin{cor}\label{cor1} Assume either (A) or (B). Let $\alpha = a$ in the case (A) and $\alpha = b$ in the case (B). Suppose $p>1$, $q>1$ and $(h,l) \in (C[0,\infty))^2 $ tal que $h(t) \sim t^r ~~(r>-1)$ ~e~ $l(t) \sim t^s ~~(s>-1)$ for $t$ large enough.
\begin{itemize}
  \item[(i)] If \, $1 < p \leq 1 + \frac{(2- \alpha)(r+1)}{N} $~ or ~ $1 < q \leq 1 + \frac{(2- \alpha)(s+1)}{N} $, then the problem $\eqref{APPL} $ has no nontrivial global solutions.
  \item[(ii)] Let $ 1 + \frac{(2- \alpha)(r+1)}{N}  < p $ ~ \, and \, ~ $1 + \frac{(2- \alpha)(s+1)}{N}  < q .$
  \begin{itemize}
    \item[(a)] If \, $ 0 < \rho \leq \max \left \{ \frac{(2 - \alpha)(r+1)}{p-1}, \frac{(2 - \alpha)(s+1)}{q-1}  \right\}$, then the problem \eqref{APPL} has no global solution for all initial data $ \psi \in I_{\rho}.$
    \item[(b)] If \, $\max \left \{ \frac{(2 - \alpha)(r+1)}{p-1}, \frac{(2 - \alpha)(s+1)}{q-1}  \right\}< \rho$, for every initial data $u_0 \in I^{\rho} $ with $0<\rho < N ,$ there exists $\Lambda(\alpha)>0$ such that the problem \eqref{APPL} with initial data $\lambda u_0 $ has nontrivial global solutions for all $ 0 < \lambda < \Lambda(\alpha) $.
  \end{itemize}
\end{itemize}
\end{cor}
\begin{rmk}\label{RMK1} Consider $ \alpha = a$ in the case $(A)$ and $\alpha=b $ in case $(B).$
\begin{itemize}
  \item[(i)] Note that the results in Corollary \ref{cor1} are sharp, since that putting $p=q$ in Corollary \ref{cor1}, we yield the following critical Fujita exponent
  $$p^\star(\alpha) = \max \left \{ 1 + \frac{(2- \alpha)(r+1)}{N}, 1 + \frac{(2- \alpha)(s+1)}{N} \right \}.$$
  Also, the second critical exponent is given by
  $$ \rho^\star(\alpha) = \max \left \{ \frac{(2 - \alpha)(r+1)}{p-1}, \frac{(2 - \alpha)(s+1)}{p-1}  \right \} .$$
\item[(ii)]It is possible to use arguments similar to those in the proof of Corollary \ref{cor1} to demonstrate other sharp results under slight modifications in Corollary \ref{cor1}. For instance, if $l(t)=0 $ and $h(t)=1$ in Corollary \ref{cor1}, we can recover the Fujita's results obtained in \cite{Fujish} (see Theorem \ref{THMFU}). In addition, in this same case, we obtained the following second critical exponent $\rho^\star(\alpha) = \frac{2 -\alpha}{p-1}.$ On the other hand, when $h(t)=0$, we can obtain the following Fujita exponent and the second critical exponent
$$ p^\star(\alpha) = 1 + \frac{(2-\alpha)(s+1)}{N} ~~\mbox{ and }~~ \rho^\star(\alpha) = \frac{(2 - \alpha)(s+1)}{p-1},  $$
respectively.
\end{itemize}
\end{rmk}

\section{Preliminaries and toolbox}

 We refer to articles \cite{Cazenave}, \cite{Fujish}, \cite{MR919445}, \cite{Gutierrez}, \cite{GutierrezArg} and the references therein for more details of what is presented in this section. Let us remember some concepts: $C[0,\infty)$ is the set of nonnegative functions defined on the interval $[0,\infty)$. $C_b(\mathbb{R}^N)$ is the set of bounded continuous functions defined on $\mathbb{R}^N$. The Banach space $L^p(\mathbb{R}^N)$ is defined as usual, and the norm is denoted by
\begin{equation*}
\| \psi \|_p = \left(\int_{\mathbb{R}^N} \, \mid \psi(x) \mid^p \, dx \right)^{1/p} $$
for $1\leq p < \infty.$ When $p=\infty $
$$\| \psi \|_\infty := \inf \left \{ K \geq 0: \mid \psi(x) \mid \leq K, \mbox{ for almost every } x \in \mathbb{R}^N \right \} .
\end{equation*} 
The set $L^\infty((0,T), C_b(\mathbb{R}^N))$ is defined as follows
\begin{equation}
L^\infty((0,T), C_b(\mathbb{R}^N))= \left\{u:(0,T)\longrightarrow C_b(\mathbb{R}^N): \|u\| = \sup_{t \in (0,T)} \| u(t) \|_{\infty} < \infty \right\}
\end{equation}
is a Banach space equipped with the usual norm $\|\cdot\|.$

We say that $f_1(t) \sim f_2 (t) $ for $ t>0$ sufficiently large, when there exist constants $k_1 >0 $, $k_2>0 $ and $t_0 > 0 $ such that
\begin{equation*}
k_1 ~ f_2(t) \leq f_1(t) \leq k_2 ~ f_2(t),
\end{equation*}
for all $t>t_0$.


The positive part of a real-valued function $\phi $ is defined by
$$\phi^{+}(t)= \max \{\phi(t), 0 \}. $$
The negative part of $\phi $ is defined analogously.

Solutions for problem \eqref{eq2.1} are understood in the following sense.
\begin{defi}
Let $u_0 \in C_b(\mathbb{R}^N).$ A function $u \in C([0,T), C_b(\mathbb{R}^N))$ is called the solution of the problem (\ref{eq2.1}) on $[0,T]$, for some $T>0$, if $u$ satisfies
\begin{equation}\label{MILD}
u(x,t) = \int_{\mathbb{R}^N} \Gamma(x,y,t) u_0(y) dy + \int_{0}^{t} \int_{\mathbb{R}^N}  \Gamma(x,y,t- \sigma) h(\sigma) f(u(y,\sigma)) d \sigma dy,
\end{equation}
for all $ t \in (0,T)$. When $T= \infty,$ we call that $ u $ is a global solution of (\ref{eq2.1}). Here, $\Gamma(x,y,t)$ be the fundamental solution of \eqref{FUND}.
\end{defi}

As the weight function $\omega(x)$  satisfies the conditions (A) or (B), it follows that $\omega(x)$ lies in the Muckenhoupt classes $A_{1 + \frac{2}{N}}$ and $A_{2}$. That is
$$
\begin{array}{ll}
 c_0 = \sup_{Q} \left( \displaystyle\fint_Q \omega(x) dx \right)\left( \displaystyle\fint_{Q} \omega (x)^{- \frac{N}{2}} dx \right)^{\frac{2}{N}} < \infty
\end{array}
 $$
and
$$ 
 C_0=\sup_{Q} \left( \fint_{Q} \omega (x) dx \right) \left( \fint_{Q} \omega(x)^{-1} dx \right) < \infty
$$ 
respectively. Where the supremum is taken over all cubes $Q$ in $\mathbb{R}^{N}$. Also, the weight function $\omega(x)^{- \frac{N}{2}} $ satisfies a doubling and reverse doubling condition of order $ \mu $ with $\mu > 1/2$. This means that there exist constants $c_1, c_2 >0 $ such that
$$ \int_{B_{sR}(x)} \omega(y)^{- \frac{N}{2}} dy \leq c_1 s^{\mu N} \int_{B_{R}(x)} \omega(y)^{- \frac{N}{2}} dy  $$
and
$$ \int_{B_{sR}(x)} \omega(y)^{- \frac{N}{2}} dy \geq c_2 s^{\mu N} \int_{B_{R}(x)} \omega(y)^{- \frac{N}{2}} dy. $$
This makes it possible for us to use the lower and upper bound estimates for the fundamental solutions of \eqref{FUND}, obtained by Guti\'{e}rrez et al. in \cite{MR919445}, \cite{Gutierrez}, which were used in paper \cite{Fujish}.
respectively.

Under the conditions (A) or (B), the fundamental solution $\Gamma = \Gamma (x, y, t)$ of \eqref{FUND} verifies the following properties.
\begin{itemize}
\item[$(\mbox{K}_{1})$] $ \displaystyle \int_{\mathbb{R}^{N}} \Gamma( x , y , t) dx=\int_{\mathbb{R}^{N}} \Gamma(x , y , t) dy=1$ for $x, y \in \mathbb{R}^{N} $ and $t > 0$;
\item[$(\mbox{K}_{2})$] $\displaystyle \Gamma(x , y , t) = \int_{\mathbb{R}^{N}} \Gamma(x, \xi , t - s) \Gamma(\xi , y , s) d\xi  $ for $x, y \in \mathbb{R}^{N} $ and $t > s > 0$;
\item[$(\mbox{K}_{3})$] There exist $ C_{0\star},c_{0\star}>0$, depending only on $N$ such that
\begin{eqnarray*}
&c_{0\star}^{-1} \left( \frac{1}{[h_{x}^{-1}(t)]^{N}} + \frac{1}{[h_{y}^{-1}(t)]^{N}} \right) e^{- c_{0\star}( \frac{h_{x}( \mid x-y \mid)}{t})^{\frac{1}{1- \alpha}}} \\
&\leq   \Gamma(x, y , t) \\
&\leq C_{0\star}^{-1} \left( \frac{1}{[h_{x}^{-1}(t)]^{N}} + \frac{1}{[h_{y}^{-1}(t)]^{N}} \right) e^{- C_{0\star} ( \frac{h_{x}( \mid x-y \mid)}{t})^{\frac{1}{1- \alpha}}}
\end{eqnarray*}
for $x, y \in \mathbb{R}^{N}$, $t > 0$, and $\alpha \in \{a, b\}$. Where $h_{x}^{-1}$ denotes the inverse function of
$$ h_x(r) = \left( \int_{B_r(x)} \omega(y)^{- \frac{N}{2}}  \right)^{\frac{2}{N}}.$$
\item[$(\mbox{K}_{4})$] $\displaystyle\int_{\mathbb{R}^N} \Gamma(x,y,t) v_0(y) dy  \in C((0,T),C_b(\mathbb{R}^N))$ is a solution of (\ref{FUND}) with initial condition $v_0 \in C_b(\mathbb{R}^N),$ and also
$$ \lim_{t \longrightarrow 0} \int_{\mathbb{R}^N} \Gamma(x,y,t) v_0(y) dy = v_0(x). $$
\item[$(\mbox{K}_{5})$] There exists a constant $ c >0 $ depending only on $N$ and $ \alpha \in \{a,b \} $, such that
$$ 0 < \Gamma(x,y,t) \leq c~ t^{ - \frac{N}{2 - \alpha}}, $$
for $x, y \in \mathbb{R}^N$ and $t > 0.$  (See \cite[p.\,10]{Fujish}).   
\end{itemize}

\begin{lem}[\cite{Fujish}] \label{PROP1}
Let $\phi \in L^{q_1}(\mathbb{R}^N)$ and $1 \leq q_1 \leq q_2 \leq \infty,$ then
\begin{equation}\label{G1}
\|S(t)\phi \|_{q_2} \leq c_1 t^{- \frac{N}{2 - \alpha} \left(\frac{1}{q_1}- \frac{1}{q_2}\right)} \|\phi\|_{q_1}, \ t>0.
\end{equation}
Where the constant $c_1$ can be taken so that it depends only on $N$ and $\alpha \in \{a, b\}.$
\end{lem}

\begin{lem}[\cite{Fujish}, Lemma 2.4] \label{Fujish2}
Assume either (A) or (B). Let $\phi \in L^{\infty}(\mathbb{R}^{N}),$ $\phi \geq 0,$ and $\phi \neq 0$.  Then there exists a positive constant $C(\alpha,N)$, depending only on $\alpha $ and $N$, such that
\begin{equation}
 S(t)\phi(x)\geq C(\alpha, N)^{-1} t^{-\frac{N}{2- \alpha}} \int_{\mid y \mid \leq t^{\frac{1}{2 - \alpha}}} \phi(y)dy,
\end{equation}
for $|x|\leq t^{\frac{1}{2- \alpha}}$ and $t>0.$ Here $\alpha = a$ in the case (A) and $\alpha = b$ in the case (B).
\end{lem}

The following Lemma will be used for the proof of non-global existence for \eqref{eq2.1}.
\begin{lem}\label{13.1} Assume either (A) or (B). If $0 \leq u_0 \in L^{\infty}(\mathbb{R}^N)$ and $f:[0,\infty) \longrightarrow [0,\infty)$ is convex, then
$$  S(t)f(u_{0}) \geq f \left( S(t)u_0 \right) $$
\end{lem}
\begin{proof}
By $(\mbox{K}_{1})$ property and the fact that $\Gamma >0$ (see $(\mbox{K}_5)$), we can use Jensen's inequality, and so we get
$$  S(t)f(u_{0}(x)) = \int_{\mathbb{R}^{N}} \Gamma(x,y,t) f(u_{0}(y)) dy \geq f \left(\int_{\mathbb{R}^{N}} \Gamma(x,y,t) u_{0}(y) dy \right)= f \left(S(t)u_0(x) \right). $$
\end{proof}

For the second critical exponent result, we need the following version of \cite[Lemma 2.12]{Lee-Ni}
\begin{prop} \label{PROP2}
Assume either (A) or (B). Let $\alpha = a$ in the case (A) and $\alpha = b$ in case (B).
\begin{itemize}
  \item[(i)] If $0 \leq u_0 \in C_b(\mathbb{R}^N) \cap L^{1}(\mathbb{R}^N)$ and $u_0 \neq 0 $, then $ \| S(t) u_0 \|_{L^{\infty}(\mathbb{R}^{N})} \sim t^{ - \frac{N}{2 - \alpha}},$ for all $ t>0 $ sufficiently large.
  \item[(ii)] If $ \psi \in I_\rho $, then there exists a constant $ C_1 > 0 $ such that
$$    \| S(t) \psi \|_{L^{\infty}(\mathbb{R}^{N})} \geq C_1~  t^{ - \frac{\rho}{2 - \alpha}}, $$
for all $t>0 $ sufficiently large.
\item[(iii)] If $ \psi \in I^\rho  \cap L^1(\mathbb{R}^N)$ , then there exists a constant $ C_2(N,\alpha) > 0 $ depending only on $\alpha $ and $N $ such that
$$    \| S(t) \psi \|_{L^{\infty}(\mathbb{R}^{N})} \leq C_2(N,\alpha)~ \|\psi\|_{L^{1}} ~ t^{ - \frac{\rho}{2 - \alpha}}, $$
for all $t>0 $ sufficiently large.
\end{itemize}
\end{prop}
\begin{proof}
Note that (i) is a direct consequence of Lemmas \ref{PROP1} and \ref{Fujish2}. Item (iii) follows from $\psi \in  L^1(\mathbb{R}^N)$ and Lemma \ref{PROP1}. Now, we prove (ii). In fact, since that $ \psi \in I_\rho $, we have that there exists $C_3>0$ such that $ \psi(x) \geq C_3 \mid x \mid^{-\rho} $ for $|x|$ sufficiently large, then by Lemma \ref{Fujish2} and $ t>0$ sufficiently large, we have
\begin{eqnarray*}
S(t)\psi(x) & \geq & C(\alpha, N)^{-1} ~ t^{-\frac{N}{2- \alpha}} \int_{\mid y \mid \leq  t^{\frac{1}{2 - \alpha}}} \psi(y)~dy  \\
&\geq & C~ t^{ - \frac{N}{2 - \alpha}} \int_{\frac{t^{\frac{1}{2 - \alpha}}}{2} \leq \mid y \mid \leq t^{\frac{1}{2 - \alpha}}}  \mid y \mid^{- \rho} ~ dy  \\
& \geq & C~ t^{ - \frac{N}{2 - \alpha}} ~ t^{ - \frac{\rho}{2 - \alpha}} ~ \int_{\frac{t^{\frac{1}{2 - \alpha}}}{2}  \leq \mid y \mid \leq t^{\frac{1}{2 - \alpha}}} ~ dy ~ \geq  ~ C t^{ - \frac{\rho}{2 - \alpha}}.
\end{eqnarray*}
Thus, the Lemma follows.
\end{proof}

\subsection{Local Existence}
This subsection proves the existence and uniqueness of local solutions of (\ref{eq2.1}) employing the known fixed point method. There exist several pieces of literature in this frame; see by example \cite{Cazenave}.
\begin{lem}{(Local Existence).}
Assume either $(A)$ or $(B)$. Let $h, l \in C[0,\infty)$ and $f,g \in C[0,\infty)$ are locally Lipschitz continuous. Then for every $u_0 \in C_b (\mathbb{R}^N)$, there exists a unique solution $u$ of problem \eqref{eq2.1} defined on $[0,T]$, for some $0<T<\infty$.
\end{lem}

\begin{proof}
Let $u_0 \in C_b(\mathbb{R}^N)$. We define $\Psi : \mathcal{K }\rightarrow L^\infty((0,T),C_b(\mathbb{R}^N))$ by
\begin{equation}
\begin{array}{ll}
\Psi(u)(t) = S(t)u_0 + \displaystyle \int_0^t S(t - \sigma) h(\sigma)f(u(\sigma)) d \sigma + \displaystyle \int_0^t S(t - \sigma) l(\sigma) g(u(\sigma)) d\sigma,
\end{array}
\end{equation}
where $\mathcal{K} = \{ u \in L^\infty((0,T); C_b(\mathbb{R}^N)), \|u(t)\|_{\infty} \leq c_1 \|u_0\|_\infty + 1, \ \forall t \in (0,T) \}$.
Note that, $\mathcal{K}$ is a complete metric space equipped with the distance induced by the norm of $ L^\infty((0,T), C_b(\mathbb{R}^N))$.

First, we show that $\Phi(u) \in \mathcal{K}$. Indeed, by \eqref{G1}, we have
\begin{equation}\label{estimate18}
\begin{array}{ll}
\|\Psi(u)\|_\infty &\leq c_1 \|u_0\|_\infty + c_1 \displaystyle \int_0^t h(\sigma) \|f(u(\sigma))\|_\infty d\sigma
\\
&+ c_1 \displaystyle \int_0^t l(\sigma)\|g(u(\sigma))\|_\infty d\sigma
\\
&\leq c_1 \|u_0\|_\infty + \max_{s \in [0,M]} \{f(s),g(s) \} \left(\sup_{s \in [0,T]} h(s) \right) c_1 \|u_0\|_\infty T
\\
&+ \max_{s \in [0,M]} \{f(s),g(s) \} \left(\sup_{s \in [0,T]} l(s) \right) (c_1 \|u_0\|_\infty + 1) T,
\\
&+ \max_{s \in [0,M]} \{f(s),g(s) \} \left(\sup_{s \in [0,T]} h(s) + \sup_{s \in [0,T]} l(s) \right)  T,
\end{array}
\end{equation}
where $M= \|u_0\|_{\infty} +1.$ Choosing $T>0$ small enough, we have $\Psi(u) \in \mathcal{K}$.

Now, arguing as in \eqref{estimate18}, we can show that for $u,v \in \mathcal{K}$,
\begin{equation}
\sup_{t \in [0,T]} \| \Psi(u)(t) - \Psi(v)(t)\|_{\infty} \leq C \left(\sup_{s \in [0,T]} h(s) + \sup_{s \in [0,T]} l(s) \right) T  \sup_{t \in [0,T]} \|u(t) - v(t)\|_\infty,
\end{equation}
therefore $\Psi$ is a strict contraction for $T$ short enough so that $$C \left(\sup_{s \in [0,T]} h(s) + \sup_{s \in [0,T]} l(s) \right) T <1.$$ Then $F $ has a unique fixed point $u \in \mathcal{K}$, i.e
\begin{equation}
u(t)= S(t)u_0 + \int_{0}^{t} S(t - \sigma)h(\sigma) f(u(\sigma)) d \sigma + + \int_{0}^{t} S(t - \sigma)h(\sigma) g(u(\sigma)) d \sigma,
\end{equation}
for $t>0$. Note that, by $(K_4)$ property and defining $u(0)=u_0$, we have that $u \in C([0,T), C_b(\mathbb{R}^N))$, and thus is a solution of (\ref{eq2.1}).

To demostrate uniqueness, suppose that $u_1, u_2 \in C([0,T), C_b(\mathbb{R}^N)) $ are solutions of \eqref{eq2.1} such that $ u_1(0) = u_2(0) = u_0$, then arguing as in \eqref{estimate18}, we have
$$ \|u_1(t) - u_2(t)  \|_{\infty}  \leq C \int_{0}^{t} \|u_1(\sigma) - u_2(\sigma)  \|_{\infty} d \sigma, \mbox{ for } t \in [0,T). $$
Thus, uniqueness follows from Gronwall's inequality.

\end{proof}

The following Lemma is a version of the comparison principle.
\begin{lem}[Comparison Principle] Assume either (A) or (B). Let $(h,l) \in (C[0,\infty))^{2} $, $(f,g) \in (C[0,\infty))^2 $ are nondecreasing locally Lipschitz functions, and $u,v \in  C([0,T], C_b(\mathbb{R}^N)) $ be solutions of problem \eqref{eq2.1}. If $ u(0) \leq v(0)$, then $u(t) \leq v(t)$ for all $t \in [0,T] $. \label{Comparison}
\end{lem}
\begin{proof} Note that it is sufficient to show that $ [u - v]^{+} \leq 0. $ Let $u(0)=u_0 $, $v(0) = v_0 $ and $M_0 = \max \{ \|u(t)\|_{\infty}, \|v(t)\|_{\infty}: t \in [0,T] \} .$ Since that $u_0\leq v_0 $, then from $(K_5)$ we have
\begin{eqnarray*}
u(t) - v(t) \leq  \int_{0}^{t} S(t- \sigma) \left(h(\sigma)[f(u(\sigma)) - f(v(\sigma))  ] + l(\sigma)[g(u(\sigma)) - g(v(\sigma))  ] \right) d\sigma.
\end{eqnarray*}
Thus, since that $ f $ and $g $ are nondecreasing and Lipschitz continuous on $[0,M_0] $, it follows from Lemma \ref{PROP1} that
\begin{eqnarray*}
\| [u(t) - v(t)]^{+} \|_{\infty} &  \leq & C \int_{0}^{t} \| [u(\sigma) - v(\sigma)]^{+} \|_{\infty}  ~ d\sigma.
\end{eqnarray*}
The Lemma is now a direct consequence of Gronwall's inequality.
\end{proof}

\section{Proof of Theorem \ref{thm1} }

\subsection*{Proof of Theorem \ref{thm1}-(i).} We apply arguments used in \cite{Weissler}. Consider $ \delta > 0 $ such that
$$ \delta < \frac{1}{ \beta + 1} \min  \left\{ 1, \frac{m}{\|v_0\|_\infty} \right\}, $$
where $ \beta >0 $ satisfies 
$$ \displaystyle \int_{0}^{\infty} h(\sigma) \frac{f \left(\|S(\sigma)v_0\|_{\infty} \right)}{\|S(\sigma)v_0\|_{\infty}} d\sigma + \int_{0}^{\infty} l(\sigma) \frac{g \left(\|S(\sigma)v_0\|_{\infty} \right)}{\|S(\sigma)v_0\|_{\infty}} d\sigma  < \frac{\beta}{\beta+1}.$$

Set $u_0 = \delta v_0 \in C_b(\mathbb{R}^N) $. Now we define the sequence $ \{ u^k \}_{k\geq0} $ by $u^0 = S(t)u_0 $ and
$$ u^k(t) = S(t)u_0 + \int_{0}^{t} S(t-\sigma) h(\sigma)f(u^{k-1}(\sigma))~ d \sigma + \int_{0}^{t} S(t-\sigma) l(\sigma)g(u^{k-1}(\sigma))~ d \sigma $$
for $k \in \mathbb{N} $ and $t>0. $

We claim
\begin{equation} \label{Claim11}
u^{k}(t) \leq (1+\beta) S(t)u_0
\end{equation}
for $k\geq0 $ and $t>0. $ We show our claim by induction on $k.$ If $k=0 $, then the claim \eqref{Claim11} is trivial. Consider now that \eqref{Claim11} holds, for some $k \in \mathbb{N}$. Since that $ \| (1+ \beta)S(t)u_0 \|_{\infty}  \leq \| S(t)v_0 \|_{\infty}   \leq m$ for $t>0 $ and $f, g, f(\cdot)/\cdot, g(\cdot)/ \cdot$ are nondecreasing on $(0,m]$, by $(K_1) $, $(K_2) $, $ (K_5)$ and Fubini's theorem we have
\begin{eqnarray*}
u^{k+1}(t) & \leq & S(t)u_0 + \int_{0}^{t} S(t - \sigma) \left[ h(\sigma) f((1+\beta) S(\sigma)u_0)  \right] ~d\sigma \\
&+& \int_{0}^{t} S(t - \sigma) \left[ l(\sigma) g((1+\beta) S(\sigma)u_0) \right] ~d\sigma \\
& = & S(t)u_0 + \int_{0}^{t} h(\sigma) S(t - \sigma) \frac{f((1+\beta) S(\sigma)u_0) }{(1 + \beta)S(\sigma)u_0}  \left[(1+\beta) S(\sigma)u_0 \right]        ~d\sigma \\
& + &  \int_{0}^{t} l(\sigma) S(t - \sigma) \frac{g((1+\beta) S(\sigma)u_0) }{(1 + \beta)S(\sigma)u_0}  \left[(1+\beta) S(\sigma)u_0 \right]        ~d\sigma \\
& \leq  &   S(t)u_0 + \left[(1+\beta) S(t)u_0 \right]  \int_{0}^{t} h(\sigma) \frac{f \left(\| S(\sigma)v_0 \|_{\infty} \right ) }{\| S(\sigma)v_0 \|_{\infty}}  ~d\sigma \\
&+& \left[(1+\beta) S(t)u_0 \right]  \int_{0}^{t} l(\sigma) \frac{g \left(\| S(\sigma)v_0 \|_{\infty} \right ) }{\| S(\sigma)v_0 \|_{\infty}}  ~d\sigma \\
& \leq & S(t)u_0 + \left[(1+\beta) S(t)u_0 \right] ~\frac{\beta}{\beta + 1}  \leq (1 + \beta) S(t)u_0.
\end{eqnarray*}
Thus, the claim follows. Also, using the fact $f$ and $g$ are nondecreasing, and by induction on $k,$ we have $ u^k \leq u^{k+1}$ for all $k \in \mathbb{N} $. Since the sequence $\{ u^k \}_{k \geq 0}$ is bounded (by \eqref{Claim11}), then by Monotone Convergence Theorem, we conclude that $ u = \lim_{k \to \infty} u^k $ is a global solution of \eqref{eq2.1}.

\subsection*{Proof of Theorem \ref{thm1}-(ii).}
First, suppose that $f \in \Phi$ and satisfies \eqref{TTa}. In order to show the non-global existence, we argue by contradiction. Suppose that there exists a global solution $ u \in C([0,\infty), C_b(\mathbb{R}^N))$ of \eqref{eq2.1} with nonnegative initial condition $u_0 \neq 0 $; this is
$$ u(t) = S(t)u_0 + \int_{0}^{t}S(t - \sigma) h(\sigma)  f(u(\sigma)) ~ d \sigma + \int_{0}^{t}S(t - \sigma) l(\sigma)  g(u(\sigma)) ~ d \sigma,$$
for $t >0$. 

Let $s>0$ be fixed and let $t>0$ such that $0<t<s $. Then, from $(K_1) $, $(K_2) $, $ (K_5)$, $(\Phi_1) $ and Lemma \ref{13.1}, we have
\begin{equation}\label{Non12}
S(s-t)u(t) \geq  \Theta(t),
\end{equation}
where $\Theta(t):= S(s)u_0 + \displaystyle\int_{0}^{t} h(\sigma) f(S(s - \sigma)u(\sigma))~d\sigma$. Note that, $\Theta(t) := \Theta(\cdot, t)  $ is absolutely continuous on $[0,s] $, consequently is differentiable almost everywhere on $[0,s] $ and
\begin{equation}\label{Non22}
 \Theta'(t) = h(t)f(S(s - t)u(t)).
\end{equation}
Since $f$ is nondecreasing, then from \eqref{Non12} and \eqref{Non22}, we have
\begin{equation}\label{Non32}
\Theta'(t)\geq h(t)f( \Theta(t) ).
\end{equation}
From $(\Phi_1) $ and $(K_5) $, this implies that
\begin{eqnarray*}
\int_{ \| S(s)u_0 \|_\infty }^{\infty} \frac{d \sigma}{f(\sigma)} &\geq & \int_{ S(s)u_0 }^{\infty} \frac{d \sigma}{f(\sigma)} ~ > ~  \int_{ \Theta(0) }^{\Theta(s)} \frac{d \sigma}{f(\sigma)} =  \int_{ \Theta(0) }^{\infty} \frac{d \sigma}{f(\sigma)} - \int_{ \Theta(s) }^{\infty} \frac{d \sigma}{f(\sigma)} \\
& = &  - \int_{0}^{s} \left(\int_{ \Theta(t) }^{\infty} \frac{d \sigma}{f(\sigma)} \right)' dt = \int_{0}^{s} \frac{\Theta'(t)}{f \left( \Theta(t) \right)}~ dt ~ \geq ~ \int_{0}^{s} h(t) ~ dt.
\end{eqnarray*}
However, this contradicts \eqref{TTa} (the same contradiction is obtained when $g \in \Phi $ and satisfies \eqref{TTb}). So that the second part of the Theorem \ref{thm1} is proved.

\section{Proof of Corollary \ref{cor1}} \label{Section3}

\subsection*{Proof of Corollary \ref{cor1}-(i).} Note that the functions $f(t) =  t^{p} ~(p>1)$ and $g(t) = (1+t)[\ln(1+t)]^q~(q>1)$ are convex functions, then the second property in $(\Phi_1) $ follows from Lemma \ref{13.1}. Also, $f$ and $g$ satisfies the integral condition in $(\Phi_1) $, since
\begin{equation} \label{Blow1}
\begin{array}{llll}
\displaystyle \int_{z}^{\infty} \frac{d \sigma}{f(\sigma)} &=& \displaystyle \int_{z}^{\infty} \frac{d \sigma}{\sigma^p} = (p-1)^{-1} z^{1-p} \, \, \, (z>0). \\
\displaystyle \int_{z}^{\infty} \frac{d \sigma}{g(\sigma)} &=& \displaystyle \int_{z}^{\infty} \frac{d \sigma}{ (1+\sigma)[\ln(1+\sigma)]^q      } = (q-1)^{-1} [\ln(1+z)]^{1-q} \, \, \, (z>0).
\end{array}
\end{equation}

Now suppose that $q < 1 + \frac{(2- \alpha)(s+1)}{N} $. From Lemma \ref{Fujish2} we have that $\| S(t)u_0 \|_{\infty} \geq c \, t^{- \frac{N}{2-\alpha}} $ for all $t$ large enough and some positive constant $c$; thus, from \eqref{Blow1}, we get
\begin{equation}\label{crit6}
\begin{array}{rll}
\displaystyle \left[ \int_{ \| S(t)u_0 \|_{\infty} }^{\infty} \frac{d \sigma}{g(\sigma)} \right]^{-1} \int_{0}^{t} l(\sigma) d \sigma &\geq & \displaystyle (q-1) [\ln(1 + \| S(t)u_0 \|_{\infty} )]^{q-1} \cdot \displaystyle \int_{0}^{t} l(\sigma) d \sigma \\
& \geq& C   \cdot  \left[ \ln(1 + c \cdot t^{-\frac{N}{2 - \alpha}}) \cdot t^{\frac{s+1}{q-1}} \right]^{q-1} \\
& \geq & C \cdot   \left[  c \, t^{- \frac{N}{2-\alpha}} \cdot t^{\frac{s+1}{q-1}} \right]^{q-1} > 1,
\end{array}
\end{equation}
for $t >0 $ sufficiently large; since $ \displaystyle \lim_{t \rightarrow \infty} \frac{\ln(1 + \, c \, t^{- \frac{N}{2}})}{t^{-\frac{N}{2}}}=c$ and $1 <q < 1 + \frac{(2- \alpha)(s+1)}{N}.$ Similarly, when $ p < 1 + \frac{(2- \alpha)(r+1)}{N}$ we have $$ \displaystyle \left[ \int_{ \| S(t)u_0 \|_{\infty} }^{\infty} \frac{d \sigma}{f(\sigma)} \right]^{-1} \int_{0}^{t} h(\sigma) d \sigma > 1, $$ for $t$ large enough. Thus, from Theorem \ref{thm1}-$(ii)$, we obtain that the solution of problem \eqref{APPL}, with initial condition $0 \leq u_0 \neq 0 $, is nonglobal.

\subsubsection*{\textbf{Critical cases $p=1+ \frac{(2-\alpha)(r+1)}{N} $ or $q=1+ \frac{(2-\alpha)(s+1)}{N}$.} } In this part, we adapt the previous arguments joint with the ideas in \cite[p.17]{Fujish}. First, consider $q=1+ \frac{(2-\alpha)(s+1)}{N};$ arguing by contradiction, we suppose that there exists a nontrivial global solution $u \in C([0,\infty), C_b(\mathbb{R}^N))$ of (\ref{APPL}), since that $u(0)=u_0 \in C_b(\mathbb{R}^N)$ is nontrivial and $l(t) \sim t^{s}~(s>-1)$ for $t$ large enough, there exists $r_0>1 $ such that $M=\displaystyle\int_{B_{r_0}(0)}u_0(y)dy \neq 0$ and $\displaystyle\int_{0}^{r_0}l(\sigma) d\sigma \neq 0.$ Thus, from Lemma \ref{Fujish2}, we have
\begin{equation}\label{Crit1}
u(x,t) \geq S(t)u_0(x) \geq C \, M \, t^{- \frac{N}{2 -\alpha}}
\end{equation}
for all $\mid x \mid \leq t^{\frac{1}{2-\alpha}}$ and $ t\geq r_0.$ We can also choose $r_0$ in such a way that
\begin{equation}\label{Crit2}
\ln(1 + CM t^{- \frac{N}{2-\alpha}}) \geq C \, M \, t^{- \frac{N}{2 -\alpha}}
\end{equation}
for all $t\geq r_0;$ since that $ \displaystyle \lim_{t \rightarrow \infty} \frac{\ln(1 + \, CM \, t^{- \frac{N}{2}})}{t^{-\frac{N}{2}}} = CM.$

From estimates $(2.11)$ and $(2.12)$ in \cite[p.10]{Fujish}, we have
\begin{equation}\label{Crit3}
u(x,t) \geq \int_{\mid x \mid \leq t^{\frac{1}{2 - \alpha}}} \Gamma(x,y,t) dx \geq  C_0>0,
\end{equation}
for all $ \mid y \mid \leq t^{\frac{1}{2 -\alpha}}.$ Since $t + r_0 - \sigma \leq t$ and $\sigma \leq t + r_0 - \sigma$ for $1 <r_0 \leq \sigma \leq t/2,$ by \eqref{Crit1}, \eqref{Crit2} and \eqref{Crit3} we have
\begin{equation}
\begin{array}{ll}
 \displaystyle \int_{\mid x \mid \leq t^{\frac{1}{2 - \alpha}}} u(x,t+r_0) \,dx  \\
\geq  \displaystyle \int_{\mid x \mid \leq t^{\frac{1}{2 - \alpha}}} \int_{r_0}^{t/2} \sigma^s \int_{\mid y \mid \leq (t+r_0 - \sigma)^{\frac{1}{2 - \alpha}}} \Gamma(x,y,t+r_0-\sigma)
\\ \noalign{\medskip}
\times (1 + u(y,\sigma))\,[\ln(1 + u(y,\sigma))]^qdyd\sigma dx  \\ \noalign{\medskip}
\geq   \displaystyle \int_{\mid x \mid \leq t^{\frac{1}{2 - \alpha}}} \int_{r_0}^{t/2} \sigma^s\int_{\mid y \mid \leq (t+r_0 - \sigma)^{\frac{1}{2 - \alpha}}} \Gamma(x,y,t+r_0-\sigma)\,
\\ \noalign{\medskip}
\times [\ln(1 + u(y,\sigma))]^q\,dy\,d\sigma\,dx  \\ \noalign{\medskip}
\geq  \displaystyle \int_{r_0}^{t/2} \sigma^s \int_{ \mid y \mid \leq (t+r_0 - \sigma)^{\frac{1}{2 - \alpha}}}\left( \displaystyle \int_{ \mid x \mid \leq t^{\frac{1}{2 - \alpha}}} \Gamma(x,y,t+r_0-\sigma)dx \right)\,
\\ \noalign{\medskip}
\times [\ln(1 + u(y,\sigma))]^q\,dy\,d\sigma  
\\ \noalign{\medskip}
\geq  C_0 \displaystyle \int_{r_0}^{t/2} \sigma^s \int_{\mid y \mid \leq \sigma^{\frac{1}{2 - \alpha}}}\,[\ln(1 + u(y,\sigma))]^q\,dy\,d\sigma \\ \noalign{\medskip}
\geq  C_0 \displaystyle \int_{r_0}^{t/2} \sigma^s \int_{\mid y \mid \leq \sigma^{\frac{1}{2 - \alpha}}}\,[\ln(1 + CM \sigma^{-\frac{N}{2-\alpha}})]^q\,dy\,d\sigma  \\ \noalign{\medskip}
\geq   C_0 \displaystyle \int_{r_0}^{t/2} \sigma^s \int_{ \mid y \mid \leq \sigma^{\frac{1}{2 - \alpha}}}\,[ CM \sigma^{-\frac{N}{2-\alpha}}]^q\,dy\,d\sigma \\ \noalign{\medskip}
\geq   C^qC_0M^q \displaystyle \int_{r_0}^{t/2} \sigma^{-\frac{N(q-1)}{2-\alpha}+s} \, \left( \int_{|y|\leq \sigma^{\frac{1}{2 - \alpha}}}\,  \sigma^{-\frac{N}{2-\alpha}} \,dy \right ) \,d\sigma \\ \noalign{\medskip}
=  C^qC_0M^q \displaystyle \int_{r_0}^{t/2} \sigma^{-1} d\sigma = C^qC_0M^q \ln\left( \frac{t}{2 r_0} \right), ~~\mbox{ for } t > 2r_0. 
\end{array}
\end{equation}
This implies that for all $m>0,$ there exist $T_m>0 $, such that
\begin{equation}\label{crit4}
\displaystyle \int_{\mid x \mid \leq T^{\frac{1}{2 - \alpha}}_{m}} u(x,T_m) \,dx \geq m.
\end{equation}
Also, from (\ref{crit4}) and Lemma (\ref{Fujish2}), we obtain
\begin{equation}\label{crit5}
\|S(t) u(T_m)\|_{\infty} \geq S(t)u(x.T_m) \geq C(\alpha,N)^{-1} m t^{-\frac{N}{2-\alpha}}
\end{equation}
for $\mid x \mid < t^{\frac{1}{2-\alpha}}$ and for all $t>T_m. $ Note that, $v(t):= u(t+T_m) $ is also a global solution of (\ref{APPL}) with initial condition $u(T_m):=u(\cdot,T_m),$ since $u $ is a global solution. Then, from (\ref{crit4}), (\ref{crit5}), arguing similarly to (\ref{crit6}), and recalling that $q=1+\frac{(2-\alpha)(s+1)}{N},$ we deduce that
\begin{eqnarray*}
\displaystyle \left[ \int_{ \| S(t)u(T_m) \|_{\infty} }^{\infty} \frac{d \sigma}{g(\sigma)} \right]^{-1} \int_{0}^{t} l(\sigma) d \sigma
&\geq&  C \cdot   \left[  C(\alpha,N)^{-1} m \, t^{- \frac{N}{2-\alpha}} \cdot t^{\frac{s+1}{q-1}} \right]^{q-1} \\
& =& C \, C(\alpha,N)^{-(q-1)} \, m^{q-1} > 1,
\end{eqnarray*}
for $m>0$ large enough. Thus, from Theorem \ref{thm1}-$(ii)$, we have that $v$ is a non-global solution which is a contradiction; thus, the solution $u$ is non-global. The proof is similar when we suppose that $p=1+\frac{(2-\alpha)(r+1)}{N}$ and omit it here. Thus, the proof is completed.

\subsection*{Proof of Corollary \ref{cor1}-(ii)-(a).} The proof is similar to the given above by using Proposition \ref{PROP2}-(ii) instead of Lemma \ref{Fujish2}.

\subsection*{Proof of Corollary \ref{cor1}-(ii)-(b).} Let a nonnegative function $ u_0 \in I^{\rho}$, $f(t) = t^p $, $g(t)= (1+t)[\ln(1+t)]^q $ and consider $\lambda >0$, which will be chosen later. Note that $f(t), g(t),\frac{f(t)}{t},\frac{g(t)}{t}$ are non-decreasing functions. Thus by Lemma \ref{PROP1} and Proposition \ref{PROP2} we have
\begin{eqnarray*}
& & \overbrace{\int_{0}^{\infty} h(\sigma) \frac{f \left(\|S(\sigma)(\lambda u_0)\|_{\infty} \right)}{\|S(\sigma)(\lambda u_0)\|_{\infty}} d\sigma + \int_{0}^{\infty} l(\sigma) \frac{g \left(\|S(\sigma)(\lambda u_0)\|_{\infty} \right)}{\|S(\sigma)(\lambda u_0)\|_{\infty}} d\sigma}^{I}  \\
&\leq & \left(\sup_{s \in [0,t_0]} h(s) \right) \int_{0}^{t_0} \left[ \lambda c_1 \| u_0\|_{\infty} \right]^{p-1}  d\sigma  +  \int_{t_0}^{\infty} \sigma^r \|S(\sigma)(\lambda u_0)\|_{\infty}^{p-1}  d\sigma \\
& & + \left(\sup_{s \in [0,t_0]} l(s) \right) \int_{0}^{t_0} (1+ \lambda c_1 \| u_0\|_{\infty})[ \lambda c_1 \| u_0\|_{\infty}  ]^{q-1}  d\sigma  \\
& & +  \int_{t_0}^{\infty} \sigma^s \, \frac{ (1+ \|S(\sigma)(\lambda u_0)\|_{\infty})[\ln(1+ \|S(\sigma)(\lambda u_0)\|_{\infty} )]^q }{ \|S(\sigma)( \lambda u_0)\|_{\infty} } d\sigma \\
&\leq & \left(\sup_{s \in [0,t_0]} h(s) \right) \int_{0}^{t_0} \left[ \lambda c_1 \| u_0\|_{\infty} \right]^{p-1}  d\sigma  
\\
& & + [C_2(N,\alpha) \lambda \|u_0\|_{L^{1}}]^{p-1} \int_{t_0}^{\infty}  \sigma^{r - \frac{ (p-1) \rho}{2 - \alpha}}  d\sigma \\
& & + \left(\sup_{s \in [0,t_0]} l(s) \right) \int_{0}^{t_0} (1+ \lambda c_1 \| u_0\|_{\infty})[ \lambda c_1 \| u_0\|_{\infty}  ]^{q-1}  d\sigma  \\
& & + (1+ \lambda c_1 \| u_0\|_{\infty}) [C_2(N,\alpha) \lambda \|u_0\|_{L^{1}}]^{q-1}  \int_{t_0}^{\infty} \sigma^{s -  \frac{(q-1) \rho }{2 - \alpha}} d\sigma \\
\end{eqnarray*}
Where $t_0>1 $ was choose such that $h(t) = t^r $ and $l(t)=t^s$ for all $t\geq t_0.$ Note that the hypothesis implies that the above integrals are bounded.   Thus, the assertion hold by Theorem \ref{thm1} and Lemma \ref{Comparison}, since $I<1 $ for $\lambda $ short enough.

\textbf{Acknowledgments.}
R. Castillo was supported by B\'io - B\'io University under grant 2020139IF/R. O. Guzm\'an - Rea was supported by CNPq/Brazil, 166685/2020-8.

\end{document}